\documentclass[11pt, reqno ]{amsart}
\usepackage{geometry}                
\usepackage{amsmath,amssymb,amsthm,mathtools}
\usepackage[T1]{fontenc}
\usepackage{lmodern,microtype}
\usepackage[colorlinks=true,linkcolor=blue,citecolor=blue,urlcolor=blue]{hyperref}
\allowdisplaybreaks[2]
\numberwithin{equation}{section}
\newtheorem{theorem}{Theorem}[section]
\newtheorem{lemma}[theorem]{Lemma}
\newtheorem{proposition}[theorem]{Proposition}

\theoremstyle{definition}
\newtheorem{definition}[theorem]{Definition}
\theoremstyle{remark}
\newtheorem{remark}[theorem]{Remark}
\newcommand{\C}{\mathbb C}
\newcommand{\N}{\mathbb Z_{\geq0}}
\newcommand{\slTwo}{\mathfrak{sl}_2}
\newcommand{\gr}{\operatorname{gr}}
\newcommand{\wt}{\operatorname{wt}}
\newcommand{\inw}{\operatorname{in}_{w}}
\newcommand{\qbinom}[2]{\begin{bmatrix}#1\\#2\end{bmatrix}_{q}}

    \advance \topmargin by -\headheight
    \advance \topmargin by -\headsep

    \evensidemargin \oddsidemargin
\numberwithin{equation}{section}

\title[Classical freeness]{Classical freeness of affine vertex algebras of $\mathfrak{sl}_2$ at boundary admissible levels}

\author{Tomoyuki Arakawa $^{1}$}

\email{tomoyuki.arakawa@oist.jp}

\author{Xuanzhong Dai $^{2}$}
\email{xuanzhong.dai@oist.jp}

\address{$^1$ Okinawa Institute of Science and Technology, Okinawa, Japan 904-0495}

\address{$^2$ School of Mathematics, Nanjing University, Nanjing 210093, China}

\begin{document}
\maketitle
\begin{abstract}
We prove that the simple affine vertex algebra $L_k(\mathfrak{sl}_2)$
is classically free at the boundary admissible level, which confirms
a conjecture of Andrews, van Ekeren, and Heluani.
\end{abstract}

\section{Introduction}
\label{sec:introduction}

Vertex algebras provide an algebraic framework for the operator
product expansions of chiral fields in two-dimensional conformal
field theory. Their representation theory is closely connected
with affine Lie algebras, modular forms, and the geometry of
singularities. A central question is how much of the structure
of a vertex algebra can be understood through its commutative
approximations. The $C_2$-algebra and the associated graded algebra
of the Li filtration offer two complementary ways to approach
this question, namely the former records relations modulo derivatives,
while the latter retains their differential structure.

Let $V$ be a finitely strongly generated vertex algebra, and set
\[
    R_V=V/C_2(V),\qquad
    C_2(V)=\operatorname{span}_{\mathbb C}
    \{a_{(-2)}b\mid a,b\in V\}.
\]
The algebra $R_V$ is a finitely generated Poisson algebra.
Li's canonical filtration $F^\bullet V$ endows
$\operatorname{gr}^{F}V$ with the structure of a Poisson vertex
algebra \cite{Li}. Since $\operatorname{gr}^{F}V$ is generated
as a differential algebra by its degree-zero component $R_V$,
there is a canonical surjective homomorphism of Poisson vertex
algebras
\begin{equation}\label{eq:intro-canonical-map}
    \pi_V:J_\infty R_V\twoheadrightarrow
    \operatorname{gr}^{F}V,
\end{equation}
where $J_\infty R_V$ denotes the arc algebra of $R_V$
\cite{AraC2}. The vertex algebra $V$ is called
\emph{classically free} if this map is an isomorphism
\cite{EH,AEH}. Equivalently, all differential relations in its
Li associated graded algebra follow from the algebraic relations
in $R_V$. In geometric terms, writing
\[
    \widetilde X_V=\operatorname{Spec}R_V,\qquad
    \operatorname{SS}(V)=
    \operatorname{Spec}\operatorname{gr}^{F}V,
\]
classical freeness identifies the singular support
$\operatorname{SS}(V)$ with the entire arc space
$J_\infty\widetilde X_V$.

Van Ekeren and Heluani showed that, for finitely strongly
generated quasiconformal vertex algebras for which
\eqref{eq:intro-canonical-map} is an isomorphism, the nodal
degeneration of the first chiral homology of an elliptic curve
is identified with the first Hochschild homology of Zhu's
algebra \cite{EH}. They also established classical freeness
for $L_\ell(\mathfrak{sl}_2)$ at nonnegative integral levels
and for the boundary Virasoro minimal models
$\operatorname{Vir}_{2,2r+1}$, whereas the remaining Virasoro
minimal models fail to be classically free. The Virasoro
examples connect the problem with Rogers--Ramanujan-type
identities and the Hilbert series of arc algebras, as developed
by Bruschek, Mourtada, and Schepers \cite{BMS}. Thus classical
freeness brings together the geometry of arc spaces, the
homology of chiral algebras, and the combinatorics of characters.

The present paper concerns the simple affine vertex algebras of $\mathfrak{sl_2}$ at boundary admissible level, i.e.
\begin{equation}\label{eq:intro-family}
    V_m=L_{k_m}(\mathfrak{sl}_2),\qquad
    k_m=-2+\frac{2}{2m+1},\qquad m\geq1.
\end{equation}
These are the nonintegral boundary admissible levels of $\widehat{\mathfrak{sl}}_2$. 
Boundary admissible representations
have particularly simple product character formulas
\cite{KW}, and their principal quantum Drinfeld--Sokolov
reductions connect them with the boundary Virasoro minimal
series. Andrews, van Ekeren, and Heluani conjectured that
every $V_m$ is classically free \cite{AEH}.
The conjecture asks whether the finite set of relations
visible in $R_{V_m}$ controls all relations among the symbols
of the affine currents and their derivatives.

An additional motivation comes from four-dimensional
$\mathcal N=2$ superconformal field theory. Beem, Lemos,
Liendo, Peelaers, Rastelli, and van Rees associated a vertex
algebra to a protected sector of local operators in such a
theory \cite{BLLPRvR}. The resulting correspondence relates
four-dimensional protected operator products to
two-dimensional chiral operator products. It also identifies
the Schur index with the vacuum supercharacter, and relates the
central charges by
\[
    c_{\mathrm{2d}}=-12c_{\mathrm{4d}},\qquad
    k_{\mathrm{2d}}=-\frac12 k_{\mathrm{4d}}.
\]
The second identity applies to the affine currents arising
from a four-dimensional flavor symmetry. This construction
makes affine vertex algebras at negative fractional levels
natural objects in the study of unitary four-dimensional
theories.

The algebras $V_m$ are proposed to be the chiral algebras of
the $(A_1,D_{2m+1})$ Argyres--Douglas theories
\cite{CS}.
These theories have $\mathfrak{sl}_2$ flavor symmetry and
central charges
\[
    c_{\mathrm{4d}}=\frac m2,\qquad
    k_{\mathrm{4d}}=\frac{8m}{2m+1}.
\]
The predicted affine level is therefore precisely $k_m$,
and the Sugawara central charge satisfies
\[
    \frac{3k_m}{k_m+2}=-6m=-12c_{\mathrm{4d}}.
\]
The agreement of Schur-index computations with affine vacuum
characters provides further evidence for this identification
\cite{CS}. Consequently, understanding relations in these
vertex algebras contributes to understanding the protected
operator spectra of an infinite family of interacting
four-dimensional theories.

Beem and Rastelli conjectured that the
Higgs branch of a four-dimensional $\mathcal N=2$
superconformal field theory is the  associated
variety of its vertex algebra \cite{BR},
\[
    \mathcal M_H\simeq X_V,\qquad
    X_V=\operatorname{Spec}\bigl((R_V)_{\mathrm{red}}\bigr).
\]
For the $(A_1,D_{2m+1})$ theories, the Higgs branch is
$\mathbb C^2/\{\pm1\}$ \cite{BR}, in agreement
with the associated variety of $V_m$ \cite{FM,AraAV},
\[
    X_{V_m}=\mathcal N(\mathfrak{sl}_2)
    \simeq
    \operatorname{Spec}
    \frac{\mathbb C[e,h,f]}{(h^2+4ef)}.
\]
Classical freeness concerns the finer scheme
$\widetilde X_{V_m}$, including its nilpotent structure.
Although reduced associated variety is the same
throughout this family, the full $C_2$-algebra retains
information about the level. The conjecture therefore
addresses a further aspect of the relation between
finite-dimensional Poisson geometry and the differential
structure of chiral operators.

Our main result is the following:
\begin{theorem}\label{thm:main}
For every integer $m\geq1$, the simple affine vertex algebra
$L_{-2+2/(2m+1)}(\mathfrak{sl}_2)$ is classically free.
Equivalently, the canonical homomorphism
\[
    J_\infty R_{V_m}\longrightarrow
    \operatorname{gr}^{F}V_m
\]
is an isomorphism of conformally graded Poisson vertex algebras.
\end{theorem}

The proof compares the Hilbert series of an explicit arc
algebra with the vacuum character. 
Let $Q=h^2+4ef$.
The singular vector calculation recalled in
\cite{AraAV} gives a surjection
\[
    A_m:=
    \mathbb C[e,h,f]/(eQ^m,hQ^m,fQ^m)
    \twoheadrightarrow R_{V_m}.
\]
We obtain an upper bound for the Hilbert series of
$J_\infty A_m$ by combining an auxiliary degeneration
with two filtrations of its graded dual. A $q$-series
identity evaluates this bound as
\begin{equation}\label{eq:intro-character}
    \operatorname{tr}_{V_m}q^{L_0}
    =
    \prod_{n\geq1}
    \frac{(1-q^{(2m+1)n})^3}{(1-q^n)^3},
\end{equation}
which coincides with the boundary admissible vacuum character \cite{KW}.
The canonical surjections supply the reverse inequality,
which proves classical freeness. Under the proposed
Argyres--Douglas identification, this also realizes
the unflavored Schur index as the Hilbert series of
the arc algebra of the full $C_2$-scheme.
As a byproduct, we obtain the following series expressions for the
vacuum character.

\begin{theorem}\label{thm:byproduct}
Let $m\geq1$ be an integer and set
$V=L_{-2+2/(2m+1)}(\mathfrak{sl}_2)$.
The unshifted ordinary character of $V$ is given by
\begin{align*}
 \operatorname{tr}_V q^{L_0}
 &=
 \frac{1}{(q)_\infty^2}
 \sum_{n_1\geq\cdots\geq n_m\geq0}
 \frac{
   q^{\sum_{j=1}^m n_j(n_j+1)}(q)_{n_m}^2
 }{
   \displaystyle
   \left(\prod_{j=1}^{m-1}(q)_{n_j-n_{j+1}}\right)
   (q)_{2n_m+1}
 }
\end{align*}
where
\[
 (q)_n=\prod_{j=1}^n(1-q^j),
 \qquad
 (q)_\infty=\prod_{j\geq1}(1-q^j),
\]
and empty products are understood to be $1$.
\end{theorem}

The paper is organized as follows. In Section \ref{sec:auxiliary-degeneration}, we  introduce some basic notations and 
the auxiliary degeneration. 
In Section \ref{sec:dual-first-filtration}, we describe the resulting graded dual by polynomial spaces and construct a
filtration by paired variables.
In Section \ref{sec:gordon-filtration}, we introduce a Gordon filtration. This filtration together  the filtration in previous section yield the required Hilbert-series bound.
In Section \ref{sec:hilbert-series-bound},  we evaluate the Hilbert-series bound using some $q$-binomial identities. 
In Section \ref{sec:vacuum-character}, we derive the vacuum character formula and complete the proof of the main theorem.

\section{The auxiliary degeneration}
\label{sec:auxiliary-degeneration}

Fix $m\geq1$, and put $p=2m+1$, $k=-2+2/p$, and $V=L_k(\slTwo)$. 
Let $R_V$ be the $C_2$-algebra
and
\begin{equation}\label{eq:C2}
 Q:=h^2+4ef.
\end{equation}
By \cite{AraAV, FM}, there is a surjective conformally graded
algebra homomorphism
\begin{equation}\label{eq:model-surjection}
    \rho_m:A_m\twoheadrightarrow R_V,
    \qquad
    A_m:=
    \mathbb{C}[e,h,f]/(eQ^m,hQ^m,fQ^m).
\end{equation}
Define the unshifted ordinary character
\begin{equation}\label{eq:character}
 H_V(q):=\operatorname{tr}_V q^{L_0}.
\end{equation}

Let
\[
 P=\C[e_i,h_i,f_i\mid i\geq0],\qquad
 E(t)=\sum_{i\geq0}e_it^i,\quad
 H(t)=\sum_{i\geq0}h_it^i,\quad
 F(t)=\sum_{i\geq0}f_it^i,
\]
and $Q(t)=H(t)^2+4E(t)F(t)$. Thus
\begin{equation}\label{eq:B}
 B_m:=J_\infty A_m=P/I_m,\qquad
 I_m=\bigl([t^n]EQ^m,[t^n]HQ^m,[t^n]FQ^m\mid n\geq0\bigr).
\end{equation}
All series coefficients are ordinary coefficients. Equivalently, the
translation derivation is $T(x_i)=(i+1)x_{i+1}$ for $x=e,h,f$.
The conformal grading is
\[
 \wt(e_i)=\wt(h_i)=\wt(f_i)=i+1.
\]
Every conformal-weight component is finite dimensional.
Passing to arc algebras in
\eqref{eq:model-surjection} and composing with the canonical
arc-algebra map gives conformally graded surjections
\begin{equation}\label{eq:arc-surjections}
    B_m
    \xrightarrow{\ J_\infty\rho_m\ }
    J_\infty R_V
    \xrightarrow{\ \pi\ }
    \operatorname{gr}^{F}V.
\end{equation}
 Consequently,
\begin{equation}\label{eq:character-lower-bound}
 H_V(q)\leq H_{J_\infty R_V}(q)\leq H_{B_m}(q),
\end{equation}
where all Hilbert-series inequalities are coefficientwise.

Equip $P$ with the auxiliary grading determined by
\[
 w(e_i)=w(f_i)=1,\qquad w(h_i)=0\qquad(i\geq0).
\]
Thus $P=\bigoplus_{d\geq0}P^w_d$, where
\begin{equation}\label{eq:auxiliary-pieces}
 P^w_d=
 \operatorname{Span}_{\C}\left\{
 \prod_{i\geq0}e_i^{\alpha_i}f_i^{\beta_i}h_i^{\gamma_i}
 \;\middle|\;
 \sum_{i\geq0}(\alpha_i+\beta_i)=d
 \right\}.
\end{equation}
The exponents in \eqref{eq:auxiliary-pieces} are nonnegative integers,
and the exponent sequences have finite support. Let
\begin{equation}\label{eq:auxiliary-filtration}
 F^w_dP=\bigoplus_{j=0}^dP^w_j,\qquad
 F^w_dB_m=(F^w_dP+I_m)/I_m\qquad(d\geq0),
\end{equation}
and put $F^w_{-1}P=F^w_{-1}B_m=0$. These are increasing
multiplicative filtrations. We write
\[
 \gr_wB_m=\bigoplus_{d\geq0}F^w_dB_m/F^w_{d-1}B_m.
\]

\begin{definition}\label{def:initial-ideal}
Let $0\neq g\in P$, and write $g=\sum_{j=0}^d g_j$ with
$g_j\in P^w_j$ and $g_d\neq0$. The initial form of $g$ with
respect to $w$ is its highest auxiliary-weight component
\[
 \inw(g)=g_d.
\]
Set $\inw(0)=0$. For an ideal $I\subseteq P$, its initial ideal
is
\[
 \inw(I)=\bigl(\inw(g)\mid g\in I\bigr)\subseteq P.
\]
\end{definition}

The definition of $\inw(I)$ uses all elements of $I$. In particular,
if $I=(g_\alpha\mid\alpha\in A)$, then
\[
 \bigl(\inw(g_\alpha)\mid\alpha\in A\bigr)
 \subseteq\inw(I),
\]
and this inclusion need not be an equality.

\begin{lemma}\label{lem:initial-associated-graded}
There is a natural isomorphism of auxiliary-graded algebras
\[
 P/\inw(I_m)\xrightarrow{\ \sim\ }\gr_wB_m.
\]
This isomorphism also preserves conformal weight.
\end{lemma}
\begin{proof}
Define a graded algebra homomorphism
\[
 \Phi:P\longrightarrow\gr_wB_m,\qquad
 \Phi(g_d)=(g_d+I_m)+F^w_{d-1}B_m
 \quad(g_d\in P^w_d).
\]
Every element of $F^w_dB_m/F^w_{d-1}B_m$ is represented by the
auxiliary-weight-$d$ component of an element of $F^w_dP$.
Hence $\Phi$ is surjective.

If $g=g_d+g_{<d}\in I_m$, with $g_d\in P^w_d$ and
$g_{<d}\in F^w_{d-1}P$, then
\[
 g_d+I_m=-g_{<d}+I_m\in F^w_{d-1}B_m.
\]
Thus $\Phi(\inw(g))=0$, and $\inw(I_m)\subseteq\ker\Phi$.
Conversely, let $0\neq g_d\in P^w_d\cap\ker\Phi$. By
\eqref{eq:auxiliary-filtration}, there is $h\in F^w_{d-1}P$ such
that $g_d+I_m=h+I_m$. It follows that
\[
 g_d-h\in I_m,\qquad \inw(g_d-h)=g_d.
\]
Consequently $g_d\in\inw(I_m)$. Since $\ker\Phi$ is homogeneous
for the auxiliary grading, we obtain $\ker\Phi=\inw(I_m)$.
Finally, the auxiliary grading and the conformal grading on $P$
are compatible, and $I_m$ is homogeneous for conformal weight.
The construction therefore preserves conformal weight.
\end{proof}

For $n\geq0$, define
\begin{equation}\label{eq:coefficient-relations}
 \begin{aligned}
  r_n^{(e)}&=[u^n]E(u)^{m+1}F(u)^m,\\
  r_n^{(f)}&=[u^n]E(u)^mF(u)^{m+1},\\
  r_n^{(h)}&=[u^n]E(u)^mF(u)^mH(u).
 \end{aligned}
\end{equation}
Here $[u^n]$ denotes extraction of the coefficient of $u^n$.
For example,
\begin{equation}\label{eq:coefficient-expanded}
 r_n^{(e)}=
 \sum_{\substack{i_1,\ldots,i_{m+1},j_1,\ldots,j_m\geq0\\
 i_1+\cdots+i_{m+1}+j_1+\cdots+j_m=n}}
 e_{i_1}\cdots e_{i_{m+1}}f_{j_1}\cdots f_{j_m}.
\end{equation}
In particular,
\[
 r_0^{(e)}=e_0^{m+1}f_0^m,\qquad
 r_1^{(e)}=(m+1)e_0^me_1f_0^m
             +me_0^{m+1}f_0^{m-1}f_1.
\]
Set
\begin{equation}\label{eq:C}
 J_m=\bigl(r_n^{(e)},r_n^{(f)},r_n^{(h)}\mid n\geq0\bigr)
 \subseteq P,\qquad C_m=P/J_m.
\end{equation}
Thus the generators of $J_m$ are the coefficient polynomials in
\eqref{eq:coefficient-relations}; individual summands of these
polynomials are not separately imposed as relations.

Recall the derivation $T$ of $P$ given by
\[
 T(e_i)=(i+1)e_{i+1},\qquad
 T(f_i)=(i+1)f_{i+1},\qquad
 T(h_i)=(i+1)h_{i+1}.
\]
For every polynomial $A\in\C[e_0,f_0,h_0]$, the Leibniz rule gives
\begin{equation}\label{eq:translation-coefficients}
 A(E(u),F(u),H(u))
 =\sum_{n\geq0}\frac{u^n}{n!}T^n A(e_0,f_0,h_0).
\end{equation}
Consequently,
\begin{equation}\label{eq:J-differential}
 J_m=\bigl(
 T^n(e_0^{m+1}f_0^m),\,
 T^n(e_0^mf_0^{m+1}),\,
 T^n(e_0^mf_0^mh_0)
 \mid n\geq0\bigr).
\end{equation}
Equivalently, $J_m$ is the smallest $T$-stable ideal containing
$(e_0f_0)^m(e_0,f_0,h_0)$. In particular,
\[
 C_m\cong
 J_\infty\left(\C[e,h,f]/\bigl((ef)^m(e,f,h)\bigr)\right).
\]

\begin{proposition}\label{prop:auxiliary-degeneration}
There is a surjective conformally graded algebra homomorphism
\[
 C_m\twoheadrightarrow\gr_wB_m.
\]
In particular,
\begin{equation}\label{eq:degeneration}
 H_{B_m}(q)\leq H_{C_m}(q).
\end{equation}
\end{proposition}
\begin{proof}
The binomial expansion gives
\[
 Q(u)^m
 =\sum_{j=0}^m\binom{m}{j}4^j
 E(u)^jF(u)^jH(u)^{2(m-j)}.
\]
The summand indexed by $j$ has auxiliary weight $2j$.
Multiplication by $E(u)$ or $F(u)$ increases this weight by $1$,
whereas multiplication by $H(u)$ leaves it unchanged. Therefore,
for every $n\geq0$,
\begin{equation}\label{eq:initial-relations}
 \begin{aligned}
  \inw\bigl([u^n]E(u)Q(u)^m\bigr)&=4^m r_n^{(e)},\\
  \inw\bigl([u^n]F(u)Q(u)^m\bigr)&=4^m r_n^{(f)},\\
  \inw\bigl([u^n]H(u)Q(u)^m\bigr)&=4^m r_n^{(h)}.
 \end{aligned}
\end{equation}
It follows from Definition~\ref{def:initial-ideal} that
\[
 J_m\subseteq\inw(I_m).
\]
Lemma~\ref{lem:initial-associated-graded} now yields
\begin{equation}\label{eq:degeneration-surjection}
 C_m=P/J_m\twoheadrightarrow
 P/\inw(I_m)\xrightarrow{\ \sim\ }\gr_wB_m.
\end{equation}
All maps in \eqref{eq:degeneration-surjection} preserve conformal
weight.

On the conformal-weight-$N$ component of $P$, the auxiliary weights
lie between $0$ and $N$. The induced filtration on $(B_m)_N$ is
therefore finite and exhaustive. Hence
\[
 \dim(B_m)_N=\dim(\gr_wB_m)_N\leq\dim(C_m)_N
 \qquad(N\geq0),
\]
which proves \eqref{eq:degeneration}.
\end{proof}

\section{The dual polynomial space and the first filtration}
\label{sec:dual-first-filtration}

\subsection{The multigrading and the restricted dual}

Equip $P$ with the $\N^3$-grading defined by
\[
 \deg(e_i)=(1,0,0),\qquad
 \deg(f_i)=(0,1,0),\qquad
 \deg(h_i)=(0,0,1).
\]
For $a,b,c\in\N$, its homogeneous component of degree $(a,b,c)$ is
\begin{equation}\label{eq:Pabc}
\begin{split}
 P_{a,b,c}
 =\operatorname{Span}_{\C}\bigl\{&
 e_{i_1}\cdots e_{i_a}
 f_{j_1}\cdots f_{j_b}
 h_{k_1}\cdots h_{k_c}\\
 &\mid i_1,\ldots,i_a,j_1,\ldots,j_b,k_1,\ldots,k_c\in\N\bigr\}.
\end{split}
\end{equation}
A basis is obtained by imposing
\[
 i_1\leq\cdots\leq i_a,\qquad
 j_1\leq\cdots\leq j_b,\qquad
 k_1\leq\cdots\leq k_c.
\]
We also use the jet grading
\[
 \deg_{\mathrm{jet}}(e_i)
 =\deg_{\mathrm{jet}}(f_i)
 =\deg_{\mathrm{jet}}(h_i)=i.
\]
Let $P_{a,b,c}[D]$ denote the span of the monomials in
\eqref{eq:Pabc} satisfying
\begin{equation}\label{eq:jet-degree-component}
 \sum_{\alpha=1}^a i_\alpha+
 \sum_{\beta=1}^b j_\beta+
 \sum_{\gamma=1}^c k_\gamma=D.
\end{equation}
Then
\[
 P=\bigoplus_{a,b,c,D\geq0}P_{a,b,c}[D],
 \qquad
 \dim P_{a,b,c}[D]<\infty,
\]
and every element of $P_{a,b,c}[D]$ has conformal weight
$D+a+b+c$. We set $P_{a,b,c}=0$ if one of $a,b,c$ is negative.

The ideal $J_m$ is homogeneous for both gradings. Define
\[
 C_{m;a,b,c}[D]
 :=
 P_{a,b,c}[D]\big/\bigl(J_m\cap P_{a,b,c}[D]\bigr),
 \qquad
 C_{m;a,b,c}:=\bigoplus_{D\geq0}C_{m;a,b,c}[D].
\]
Their restricted duals are
\[
 P_{a,b,c}^{\vee}:=\bigoplus_{D\geq0}P_{a,b,c}[D]^*,
 \qquad
 C_{m;a,b,c}^{\vee}:=\bigoplus_{D\geq0}C_{m;a,b,c}[D]^*.
\]
The quotient map identifies $C_{m;a,b,c}^{\vee}$ with the
annihilator of $J_m\cap P_{a,b,c}$ in $P_{a,b,c}^{\vee}$.

We denote by $\mathfrak S_i$ the $i$-th symmetric group for any $i\in \mathbb Z_{\geq 0}$, where as convention $\mathfrak S_0$ is understood to be trivial group. Let
\[
 \mathcal S_{a,b,c}
 :=
 \C[x_1,\ldots,x_a,y_1,\ldots,y_b,z_1,\ldots,z_c]
 ^{\mathfrak S_a\times\mathfrak S_b\times\mathfrak S_c},
\]
where all polynomial variables have degree one and empty variable families are allowed.
Define
\begin{equation}\label{eq:dual-polynomial-map}
\begin{split}
 \iota_{a,b,c}:P_{a,b,c}^{\vee}&\longrightarrow\mathcal S_{a,b,c},\\
 \lambda&\longmapsto p_\lambda(\boldsymbol x;\boldsymbol y;\boldsymbol z)
 :=\lambda\left(
 \prod_{\alpha=1}^a E(x_\alpha)
 \prod_{\beta=1}^b F(y_\beta)
 \prod_{\gamma=1}^c H(z_\gamma)\right).
\end{split}
\end{equation}
If $\lambda\in P_{a,b,c}[D]^*$ is extended by zero on the other
jet degrees, then $p_\lambda$ is homogeneous of polynomial degree $D$.
In particular, \eqref{eq:dual-polynomial-map} takes values in
polynomials.

\begin{lemma}\label{lem:dual-polynomial-identification}
The map $\iota_{a,b,c}$ is an isomorphism of graded vector spaces.
\end{lemma}
\begin{proof}
Commutativity of $P$ implies separate symmetry of $p_\lambda$.
For $\boldsymbol i\in\N^a$, $\boldsymbol j\in\N^b$ and
$\boldsymbol k\in\N^c$, coefficient extraction gives
\begin{equation}\label{eq:dual-polynomial-coefficients}
 [\boldsymbol x^{\boldsymbol i}
  \boldsymbol y^{\boldsymbol j}
  \boldsymbol z^{\boldsymbol k}]p_\lambda
 =
 \lambda\left(
 \prod_{\alpha=1}^a e_{i_\alpha}
 \prod_{\beta=1}^b f_{j_\beta}
 \prod_{\gamma=1}^c h_{k_\gamma}\right).
\end{equation}
Thus $\iota_{a,b,c}$ is injective. Conversely, the coefficients of a
separately symmetric polynomial define a functional on the monomial
basis of $P_{a,b,c}$ by \eqref{eq:dual-polynomial-coefficients}.
Separate symmetry makes this assignment independent of the ordering
of indices within each family. A polynomial has only finitely many
homogeneous components, so the resulting functional belongs to
$P_{a,b,c}^{\vee}$. This construction is inverse to $\iota_{a,b,c}$.
\end{proof}

\subsection{The defining relations and diagonal vanishing}

Set
\[
 W_{a,b,c}:=\iota_{a,b,c}(C_{m;a,b,c}^{\vee})
 \subseteq\mathcal S_{a,b,c}.
\]
Recall the generators $r_n^{(e)},r_n^{(f)},r_n^{(h)}$ of $J_m$
defined in Section~\ref{sec:auxiliary-degeneration}. Their
multidegrees are $(m+1,m,0)$, $(m,m+1,0)$ and $(m,m,1)$,
respectively. It follows that
\begin{equation}\label{eq:ideal-multidegree-component}
\begin{split}
 J_m\cap P_{a,b,c}
 =\sum_{n\geq0}\bigl(
 &r_n^{(e)}P_{a-m-1,b-m,c}
 +r_n^{(f)}P_{a-m,b-m-1,c}\\
 &+r_n^{(h)}P_{a-m,b-m,c-1}\bigr).
\end{split}
\end{equation}

\begin{proposition}\label{prop:dual-diagonals}
The subspace $W_{a,b,c}$ consists precisely of the polynomials
$p\in\mathcal S_{a,b,c}$ satisfying
\begin{equation}\label{eq:forbidden}
\begin{aligned}
 \left.p\right|_{x_1=\cdots=x_{m+1}=y_1=\cdots=y_m}&=0,\\
 \left.p\right|_{x_1=\cdots=x_m=y_1=\cdots=y_{m+1}}&=0,\\
 \left.p\right|_{x_1=\cdots=x_m=y_1=\cdots=y_m=z_1}&=0.
\end{aligned}
\end{equation}
Each condition is imposed only when all the indicated variables
exist. By separate symmetry, the same conditions hold for every
choice of the indicated numbers of variables.
\end{proposition}
\begin{proof}
Write $p=p_\lambda$ using Lemma~\ref{lem:dual-polynomial-identification}.
Assume first that $a\geq m+1$ and $b\geq m$, and put
$A=a-m-1$, $B=b-m$. Denote the residual $x$- and $y$-variables by
$\boldsymbol x'=(x_{m+2},\ldots,x_a)$ and
$\boldsymbol y'=(y_{m+1},\ldots,y_b)$. The first specialization is
\begin{equation}\label{eq:first-diagonal-generating-function}
\begin{split}
 &p_\lambda(
 \underbrace{u,\ldots,u}_{m+1},\boldsymbol x';
 \underbrace{u,\ldots,u}_{m},\boldsymbol y';
 \boldsymbol z)\\
 &\quad=
 \lambda\left(
 E(u)^{m+1}F(u)^m
 \prod_{\alpha=1}^{A}E(x'_\alpha)
 \prod_{\beta=1}^{B}F(y'_\beta)
 \prod_{\gamma=1}^{c}H(z_\gamma)\right).
\end{split}
\end{equation}
Since $E(u)^{m+1}F(u)^m=\sum_{n\geq0}r_n^{(e)}u^n$,
the coefficient of
$u^n(\boldsymbol x')^{\boldsymbol i}
(\boldsymbol y')^{\boldsymbol j}\boldsymbol z^{\boldsymbol k}$
in \eqref{eq:first-diagonal-generating-function} is
\begin{equation}\label{eq:relation-multiple-coefficient}
 \lambda\left(
 r_n^{(e)}
 \prod_{\alpha=1}^{A}e_{i_\alpha}
 \prod_{\beta=1}^{B}f_{j_\beta}
 \prod_{\gamma=1}^{c}h_{k_\gamma}\right).
\end{equation}
The monomials following $r_n^{(e)}$ span $P_{A,B,c}$.
Consequently, the first diagonal condition is equivalent to
\[
 \lambda\left(\sum_{n\geq0}
 r_n^{(e)}P_{a-m-1,b-m,c}\right)=0.
\]
If $a<m+1$ or $b<m$, this subspace is zero by our convention,
and the first family of relations imposes no condition.

The remaining specializations insert the series
\[
 E(u)^mF(u)^{m+1}
 \qquad\text{and}\qquad
 E(u)^mF(u)^mH(u),
\]
respectively. The same coefficient calculation shows that the
second and third conditions are equivalent to
\[
 \lambda\left(\sum_{n\geq0}
 r_n^{(f)}P_{a-m,b-m-1,c}\right)=0,
 \qquad
 \lambda\left(\sum_{n\geq0}
 r_n^{(h)}P_{a-m,b-m,c-1}\right)=0,
\]
respectively. By \eqref{eq:ideal-multidegree-component}, the three
conditions together are equivalent to
$\lambda(J_m\cap P_{a,b,c})=0$, as required.
\end{proof}

\subsection{The paired-variable filtration}

Fix $a,b,c\in\N$ and put $M=\min(a,b)$. For $0\leq R\leq M$,
define
\[
 \mathcal A_R:=
 \C[\tau_1,\ldots,\tau_R,x_{R+1},\ldots,x_a,
 y_{R+1},\ldots,y_b,z_1,\ldots,z_c]
 ^{\mathfrak S_R\times\mathfrak S_{a-R}
 \times\mathfrak S_{b-R}\times\mathfrak S_c}
\]
and the specialization map
\begin{equation}\label{eq:pair-specialization}
\begin{split}
 \phi_R:\mathcal S_{a,b,c}&\longrightarrow\mathcal A_R,\\
 p&\longmapsto
 p(\tau_1,\ldots,\tau_R,x_{R+1},\ldots,x_a;
   \tau_1,\ldots,\tau_R,y_{R+1},\ldots,y_b;
   \boldsymbol z).
\end{split}
\end{equation}
For this specialization, the \emph{pair variables} are
$\tau_1,\ldots,\tau_R$, and the \emph{residual variables} are
\[
 x_{R+1},\ldots,x_a,\qquad
 y_{R+1},\ldots,y_b,\qquad
 z_1,\ldots,z_c.
\]
Whenever we identify a selected collection of polynomial variables,
we call the original variables not involved in that identification
the \emph{residual variables}. In the paired-variable filtration
below, these are precisely the variables left unpaired. Further
identifications among the pair variables leave the residual
families unchanged.
Thus the three residual families have sizes $a-R$, $b-R$, and $c$;
in particular, every $z$-variable is residual.
Here the $\tau_i$ are independent variables. Symmetry in
$\tau_1,\ldots,\tau_R$ follows by simultaneously permuting the
first $R$ variables in the $x$- and $y$-families of $p$.
Permutations within each residual family also preserve the image.
Thus \eqref{eq:pair-specialization} has the stated target.
All variables in $\mathcal A_R$ have polynomial degree one,
and $\phi_R$ preserves polynomial degree.

Set
\begin{equation}\label{eq:first-filtration-kernels}
 K_R:=\ker(\phi_R|_{W_{a,b,c}})\quad(0\leq R\leq M),
 \qquad
 K_{M+1}:=W_{a,b,c}.
\end{equation}

\begin{lemma}\label{lem:first-filtration}
The spaces in \eqref{eq:first-filtration-kernels} form an increasing
graded filtration
\begin{equation}\label{eq:first-filtration}
 0=K_0\subseteq K_1\subseteq\cdots\subseteq K_M
 \subseteq K_{M+1}=W_{a,b,c}.
\end{equation}
For each $0\leq R\leq M$, $\phi_R$ induces a graded isomorphism
\begin{equation}\label{eq:first-filtration-image}
 \overline\phi_R:K_{R+1}/K_R
 \xrightarrow{\ \sim\ }\mathcal U_R:=\phi_R(K_{R+1})
 \subseteq\mathcal A_R,
 \qquad
 p+K_R\longmapsto\phi_R(p).
\end{equation}
\end{lemma}
\begin{proof}
The map $\phi_0$ is the identity, so $K_0=0$. For $R<M$,
\begin{equation}\label{eq:successive-pair-specialization}
 \phi_{R+1}(p)
 =
 \left.\phi_R(p)\right|_{x_{R+1}=y_{R+1}=\tau_{R+1}}.
\end{equation}
It follows that $K_R\subseteq K_{R+1}$.
The inclusion $K_M\subseteq K_{M+1}$ holds by definition.
Since every specialization preserves polynomial degree, all
these subspaces are graded. Finally,
\[
 \ker(\phi_R|_{K_{R+1}})
 =K_{R+1}\cap K_R=K_R,
\]
which proves \eqref{eq:first-filtration-image}.
\end{proof}

For a fixed $R$, put $s=a-R$ and $t=b-R$, and relabel the residual
$x$- and $y$-variables as $x_1,\ldots,x_s$ and $y_1,\ldots,y_t$.
Thus each $g\in\mathcal U_R$ is of the form
\begin{equation}\label{eq:paired-image-polynomial}
 g(\boldsymbol\tau;\boldsymbol x;\boldsymbol y;\boldsymbol z)
 =
 p(\tau_1,\ldots,\tau_R,\boldsymbol x;
   \tau_1,\ldots,\tau_R,\boldsymbol y;
   \boldsymbol z),
 \qquad p\in K_{R+1}.
\end{equation}
In particular, each pair variable $\tau_i$ replaces one original
$x$-variable and one original $y$-variable. The integer $t=b-R$
counts residual variables and is distinct from the variables $\tau_i$.

\begin{lemma}\label{lem:paired-image-conditions}
Every $g\in\mathcal U_R$ satisfies the following conditions:
\begin{enumerate}
 \item If $R\geq m+1$, then
 \[
 \left.g\right|_{\tau_1=\cdots=\tau_{m+1}}=0.
 \]
 \item If $R\geq m$, then
 \[
 \left.g\right|_{\tau_1=\cdots=\tau_m=\xi}=0
 \qquad
 \bigl(\xi\in\{x_1,\ldots,x_s,y_1,\ldots,y_t,z_1,\ldots,z_c\}\bigr).
 \]
 \item With the convention that an empty product is $1$,
 \[
 \Delta_{s,t}:=\prod_{\alpha=1}^s\prod_{\beta=1}^t
 (x_\alpha-y_\beta)
 \quad\text{divides }g.
 \]
\end{enumerate}
The first two conditions hold for any choice of the indicated
number of pair variables.
\end{lemma}
\begin{proof}
Choose $p\in K_{R+1}$ as in \eqref{eq:paired-image-polynomial}.
Since $p\in W_{a,b,c}$, it satisfies \eqref{eq:forbidden}.

If $\tau_1=\cdots=\tau_{m+1}=u$, then the corresponding
$m+1$ original $x$-variables and $m+1$ original $y$-variables
are all equal to $u$. The first condition in \eqref{eq:forbidden}
already gives vanishing on the diagonal with $m+1$ of these
$x$-variables and $m$ of these $y$-variables identified.
It remains valid when the additional $y$-variable is also set
equal to $u$. This proves (1).

For (2), identifying $\tau_1,\ldots,\tau_m$ identifies $m$
original variables of each of the first two families. Setting
their common value equal to a residual variable gives the
following numbers of coincident original variables:
\[
 \begin{array}{c|ccc}
 \text{residual variable}&x&y&z\\ \hline
 x_\alpha&m+1&m&0\\
 y_\beta&m&m+1&0\\
 z_\gamma&m&m&1
 \end{array}.
\]
The corresponding assertion follows from the first, second,
or third condition of \eqref{eq:forbidden}, respectively.
Symmetry in the $\tau_i$ gives both (1) and (2) for every
choice of pair variables.

Suppose $R<M$. Then $s,t>0$, and $p\in K_{R+1}$ implies
$\phi_{R+1}(p)=0$. By \eqref{eq:successive-pair-specialization},
\begin{equation}\label{eq:residual-diagonal-vanishing}
 \left.g\right|_{x_1=y_1=u}=0.
\end{equation}
Separate symmetry in the residual families yields
$\left.g\right|_{x_\alpha=y_\beta}=0$ for every
$1\leq\alpha\leq s$ and $1\leq\beta\leq t$.
Hence $x_\alpha-y_\beta$ divides $g$ for every such pair.
These linear polynomials are pairwise nonassociate irreducibles,
so unique factorization implies $\Delta_{s,t}\mid g$.
If $R=M$, at least one of $s,t$ is zero, and $\Delta_{s,t}=1$.
This proves (3).
\end{proof}

The polynomial $\Delta_{s,t}$ is invariant under the permutation
group defining $\mathcal A_R$. Hence the quotient
$g/\Delta_{s,t}$ also belongs to $\mathcal A_R$, and
\[
 \mathcal U_R\subseteq\Delta_{s,t}\mathcal A_R.
\]
Finally, taking homogeneous components in
\eqref{eq:first-filtration} and using
\eqref{eq:first-filtration-image} gives
\begin{equation}\label{eq:first-filtration-dimensions}
 \dim C_{m;a,b,c}[D]
 =\dim W_{a,b,c}[D]
 =\sum_{R=0}^{M}\dim\mathcal U_R[D]
 \qquad(D\geq0).
\end{equation}
We next filter each $\mathcal U_R$ with respect to coincidences
among the pair variables.

\section{The Gordon filtration and its multiplicities}
\label{sec:gordon-filtration}

We use the Gordon filtration (e.g. see
\cite{FJMMT}). We give the evaluation maps and
the multiplicity argument explicitly, including the contribution
of the residual variables.

Fix $R,s,t,c\geq0$, and recall that
\[
 \mathcal A_R=
 \C[\tau_1,\ldots,\tau_R,
       \boldsymbol x,\boldsymbol y,\boldsymbol z]
 ^{\mathfrak S_R\times\mathfrak S_s\times
   \mathfrak S_t\times\mathfrak S_c}.
\]
Let $U\subseteq\mathcal A_R$ be a graded $\C$-linear subspace
whose elements satisfy the three conditions in
Lemma~\ref{lem:paired-image-conditions}. We shall apply the
construction to $U=\mathcal U_R$.
The grading in this section is polynomial degree: every pair,
block, and residual variable has degree one.
All evaluations leave the residual variables unchanged.
Thus they may be regarded as coefficient parameters in a
polynomial ring over $\C$; no module structure on $U$ over
that coefficient ring is assumed.

\subsection{Block evaluations and the filtration}

Let $\lambda=(\lambda_1,\ldots,\lambda_\ell)\vdash R$, with
$\lambda_1\geq\cdots\geq\lambda_\ell>0$. Define
\begin{equation}\label{eq:block-evaluation}
 \psi_\lambda(f)(v_1,\ldots,v_\ell;
                  \boldsymbol x,\boldsymbol y,\boldsymbol z)
 :=f(\underbrace{v_1,\ldots,v_1}_{\lambda_1},\ldots,
       \underbrace{v_\ell,\ldots,v_\ell}_{\lambda_\ell};
       \boldsymbol x,\boldsymbol y,\boldsymbol z).
\end{equation}
Each group of identified pair variables is called a block,
and $v_j$ is its block variable. Symmetry in the pair variables
implies that the evaluation is independent of the choice of
disjoint blocks with these sizes, up to relabeling their variables.
In particular, it is symmetric in the variables belonging to
blocks of the same size.

Order the partitions of $R$ decreasingly in lexicographic order.
Thus $\mu>\lambda$ if the first unequal part of $\mu$ is larger
than the corresponding part of $\lambda$; partitions are padded
with zeros when necessary. Set
\begin{equation}\label{eq:gordon-subspaces}
 U_\lambda:=\bigcap_{\mu>\lambda}\ker(\psi_\mu|_U).
\end{equation}
The intersection over the empty set is understood to be $U$.
To describe the filtration precisely, write the partitions as
$\lambda^{(1)}>\cdots>\lambda^{(L)}$. Then
\[
 U_{\lambda^{(j)}}
 =\bigcap_{i<j}\ker(\psi_{\lambda^{(i)}}|_U),
 \qquad
 U_{\lambda^{(j+1)}}
 =U_{\lambda^{(j)}}\cap\ker\psi_{\lambda^{(j)}}
 \quad(j<L).
\]
The last partition is $(1^R)$, whose evaluation is the identity
up to renaming variables. Hence, after adjoining the terminal
subspace $0$, we obtain the finite decreasing graded filtration
\begin{equation}\label{eq:gordon-filtration}
 U=U_{\lambda^{(1)}}\supseteq\cdots
 \supseteq U_{\lambda^{(L)}}\supseteq0.
\end{equation}

For each partition define the successive quotient
\begin{equation}\label{eq:gordon-quotient}
 \mathcal Q_\lambda
 :=U_\lambda/(U_\lambda\cap\ker\psi_\lambda).
\end{equation}
The evaluation induces the graded isomorphism onto its image
\begin{equation}\label{eq:gordon-image}
 \overline\psi_\lambda:\mathcal Q_\lambda
 \xrightarrow{\ \sim\ }\psi_\lambda(U_\lambda),
 \qquad [f]\longmapsto\psi_\lambda(f),
\end{equation}
as its kernel on $U_\lambda$ is exactly the denominator
in \eqref{eq:gordon-quotient}. For $R=0$, the unique partition
is the empty partition with evaluation the identity, and
the filtration consists of $U\supseteq0$.

If a part of $\lambda$ is larger than $m$, its evaluation
identifies at least $m+1$ pair variables. Condition~(1) of
Lemma~\ref{lem:paired-image-conditions} then gives
$\psi_\lambda(U)=0$, and $\mathcal Q_\lambda=0$.
We may therefore restrict attention to partitions
\begin{equation}\label{eq:block-multiplicities}
 \lambda=(m^{r_m},\ldots,1^{r_1}),\qquad
 R=\sum_{i=1}^m i r_i.
\end{equation}
Here $r_i$ is the number of parts equal to $i$; the notation
$i^{r_i}$ denotes repetition of the part $i$.
Write $u_1^{(i)},\ldots,u_{r_i}^{(i)}$ for the variables of
the blocks of size $i$. Each $u_\alpha^{(i)}$ has polynomial
degree one, regardless of the size $i$ of its block.

\begin{remark}\label{rem:gordon-example}
For $R=4$ and $m=2$, the partitions are ordered as
\[
 (4)>(3,1)>(2,2)>(2,1,1)>(1,1,1,1).
\]
The first two evaluations vanish on $U$. The remaining
filtration is
\[
 U=U_{(2,2)}
 \supseteq U_{(2,1,1)}=\ker(\psi_{(2,2)}|_U)
 \supseteq U_{(1^4)}
 =\ker(\psi_{(2,2)}|_U)\cap\ker(\psi_{(2,1,1)}|_U)
 \supseteq0.
\]
For example, suppressing the residual variables,
$\psi_{(2,2)}(f)=f(u,u,v,v)$ and
$\psi_{(2,1,1)}(f)=f(u,u,v,w)$.
\end{remark}

\subsection{Taylor coefficients and cluster multiplicities}

We first record the polynomial dependence on the number of
shifted variables that will be used in the multiplicity argument.

\begin{lemma}\label{lem:taylor-count}
Let $A$ be a polynomial ring over $\C$, let $N\geq1$, and let
$F\in A[T_1,\ldots,T_N]^{\mathfrak S_N}$.
For integers $0\leq v\leq N$, expand
\begin{equation}\label{eq:shifted-evaluation}
 F(\underbrace{\zeta+\delta,\ldots,\zeta+\delta}_{v},
   \underbrace{\zeta,\ldots,\zeta}_{N-v})
 =\sum_{d\geq0}P_d(v)\delta^d.
\end{equation}
For each $d$, these coefficients are the values at
$v=0,\ldots,N$ of a polynomial $P_d(V)\in A[\zeta][V]$
of degree at most $d$ in $V$.
\end{lemma}
\begin{proof}
For each fixed integer $v$, the left-hand side of
\eqref{eq:shifted-evaluation} is a polynomial in $\delta$,
so the expansion exists and is finite. Its coefficient of
$\delta^d$ is given by the multivariable Taylor formula:
\begin{equation}\label{eq:taylor-multiindex}
 P_d(v)=
 \sum_{\substack{\alpha_1,\ldots,\alpha_v\geq0\\
                   \alpha_1+\cdots+\alpha_v=d}}
 \frac{1}{\alpha_1!\cdots\alpha_v!}
 \left.
 \frac{\partial^d F}
      {\partial T_1^{\alpha_1}\cdots\partial T_v^{\alpha_v}}
 \right|_{T_1=\cdots=T_N=\zeta}.
\end{equation}
For $d\geq1$, group these summands according to the support
$\{i:\alpha_i>0\}$, and let $j$ be its cardinality.
Necessarily $1\leq j\leq\min(d,v)$.
Symmetry of $F$, together with evaluation on the full diagonal,
implies that the total contribution from a fixed support of
size $j$ is independent of that support. It equals
\[
 c_{d,j}:=
 \sum_{\substack{\beta_1,\ldots,\beta_j\geq1\\
                 \beta_1+\cdots+\beta_j=d}}
 \frac{1}{\beta_1!\cdots\beta_j!}
 \left.
 \frac{\partial^d F}
      {\partial T_1^{\beta_1}\cdots\partial T_j^{\beta_j}}
 \right|_{T_1=\cdots=T_N=\zeta}
 \in A[\zeta].
\]
There are $\binom vj$ such supports. Therefore the polynomial
\begin{equation}\label{eq:taylor-binomial}
 P_d(V):=\sum_{j=1}^{\min(d,N)}c_{d,j}\binom Vj,
 \qquad
 \binom Vj=\frac{V(V-1)\cdots(V-j+1)}{j!},
 \quad d\geq1,
\end{equation}
has the required values, since $\binom vj=0$ for integers
$0\leq v<j$. Each summand has degree at most $j\leq d$.
For $d=0$, take the constant polynomial
$P_0(V)=F(\zeta,\ldots,\zeta)$.
\end{proof}

\begin{lemma}\label{lem:cluster}
Let $\lambda$ be as in \eqref{eq:block-multiplicities} and
$f\in U_\lambda$. For every two distinct blocks,
\[
 (u_\alpha^{(i)}-u_\beta^{(j)})^{2\min(i,j)}
 \quad\text{divides}\quad\psi_\lambda(f).
\]
When $i=j$, distinctness means $\alpha\ne\beta$.
\end{lemma}
\begin{proof}
Choose the two blocks, write their sizes as $a\geq b$, and put
$N=a+b$. First evaluate all other blocks, leaving the $N$
pair variables in the chosen two blocks independent. The
result is a polynomial $F(T_1,\ldots,T_N)$ symmetric in those
$N$ variables. Its coefficient ring $A$ is the polynomial
ring in the other block variables and all residual variables.
Apply Lemma~\ref{lem:taylor-count} to $F$.

For an integer $0\leq v\leq N$, let $\mu(v)$ be the partition
obtained from $\lambda$ by replacing the two chosen parts
$(a,b)$ by $(N-v,v)$, omitting a zero part and sorting the
result. We claim that
\begin{equation}\label{eq:outer-root-values}
 \mu(v)>\lambda
 \qquad\text{if}\qquad
 v\in\{0,\ldots,b-1\}\cup\{N-b+1,\ldots,N\}.
\end{equation}
Indeed, in the first range $N-v\geq a+1$, while in the
second range $v\geq a+1$. Thus the larger new part
$h=\max(N-v,v)$ is strictly larger than $a$.
All parts larger than $h$ are unchanged, and the number of
parts equal to $h$ increases by one, since neither removed
part is that large. The first difference in the sorted
partitions is consequently an increase. This proves
\eqref{eq:outer-root-values}, including the cases $v=0,N$,
where the two blocks have merged into a single block.

For each $v$ in \eqref{eq:outer-root-values}, the left-hand
side of \eqref{eq:shifted-evaluation} is
$\psi_{\mu(v)}(f)$ with its block variables relabeled and
specialized to the indicated values. Since $f\in U_\lambda$,
this evaluation is zero. Hence, for every $d$, the polynomial
$P_d(V)$ vanishes at all those values of $v$.
Each range has $b$ elements, and the ranges are disjoint
because $a\geq b$. Thus $P_d(V)$ has $2b$ distinct roots.
It has degree at most $d$ by Lemma~\ref{lem:taylor-count},
so
\[
 P_d(V)=0\qquad(0\leq d<2b).
\]
The root count can be taken in the fraction field of
$A[\zeta]$, so it applies to these polynomial coefficients.

Now take $v=a$ in \eqref{eq:shifted-evaluation}. If $u,w$
denote the variables of the chosen blocks of sizes $a,b$,
respectively, this gives
\[
 \left.\psi_\lambda(f)\right|_{u=\zeta+\delta,\,w=\zeta}
 \in\delta^{2b}A[\zeta,\delta].
\]
The substitution $u=\zeta+\delta$, $w=\zeta$ is an invertible
change of polynomial variables, with inverse
$\zeta=w$, $\delta=u-w$. Therefore
$(u-w)^{2b}$ divides $\psi_\lambda(f)$, as required.
\end{proof}

\begin{remark}\label{rem:equal-blocks}
The preceding argument includes two distinct blocks of the
same size $i$. In this case $a=b=i$, $N=2i$, and the root
values are
\[
 0,1,\ldots,i-1,\qquad i+1,\ldots,2i.
\]
The remaining value $v=i$ gives the original two blocks.
Their collision therefore has multiplicity at least $2i$.
These factors must be included even though the two blocks
belong to the same block-size family. Symmetry in that
family alone would not establish this multiplicity.
\end{remark}

\subsection{The common divisor}

Define the cluster divisor by
\begin{equation}\label{eq:cluster-divisor}
 D_{\mathrm{cl},\lambda}:=
 \prod_{i=1}^m\prod_{1\leq\alpha<\beta\leq r_i}
 (u_\alpha^{(i)}-u_\beta^{(i)})^{2i}
 \prod_{1\leq i<j\leq m}
 \prod_{\alpha=1}^{r_i}\prod_{\beta=1}^{r_j}
 (u_\alpha^{(i)}-u_\beta^{(j)})^{2i}.
\end{equation}
The first product runs over pairs of distinct blocks of the
same size, and the second over pairs of different sizes.
Thus every unordered pair of distinct blocks occurs exactly
once, with exponent $2\min(i,j)$.

There are two further families of factors. Set
\begin{equation}\label{eq:size-m-residual-divisor}
 D_{m,\mathrm{res},\lambda}:=
 \prod_{\alpha=1}^{r_m}
 \left(
  \prod_{\beta=1}^s(u_\alpha^{(m)}-x_\beta)
  \prod_{\beta=1}^t(u_\alpha^{(m)}-y_\beta)
  \prod_{\beta=1}^c(u_\alpha^{(m)}-z_\beta)
 \right).
\end{equation}
Recall also the divisor in the residual variables
$\Delta_{s,t}=\prod_{\alpha=1}^s\prod_{\beta=1}^t
(x_\alpha-y_\beta)$ from
Lemma~\ref{lem:paired-image-conditions}.
The full divisor is
\begin{equation}\label{eq:divisor}
 D_\lambda:=D_{\mathrm{cl},\lambda}
             D_{m,\mathrm{res},\lambda}\Delta_{s,t}.
\end{equation}
All products over empty index sets are $1$.

\begin{proposition}\label{prop:gordon-divisor}
For every $f\in U_\lambda$, the polynomial $\psi_\lambda(f)$
is divisible by $D_\lambda$.
\end{proposition}
\begin{proof}
Write $G=\psi_\lambda(f)$. Lemma~\ref{lem:cluster} supplies
each factor in \eqref{eq:cluster-divisor}, with its stated
multiplicity.

Consider a block of size $m$ with variable $u_\alpha^{(m)}$.
Setting $u_\alpha^{(m)}=x_\beta$ makes the $m$ original
pair variables in that block equal to the residual variable
$x_\beta$. Condition~(2) of
Lemma~\ref{lem:paired-image-conditions} implies that the
resulting evaluation is zero, also after evaluating the
other blocks. The polynomial factor theorem therefore gives
$u_\alpha^{(m)}-x_\beta\mid G$. The same argument applies
to $y_\beta$ and $z_\beta$, giving every factor in
\eqref{eq:size-m-residual-divisor}.
Here it is specifically the block size $m$ that is used:
if $r_m=0$, this product is $1$, even if some smaller size
is the largest part present in $\lambda$.

Finally, condition~(3) gives $f=\Delta_{s,t}f_0$ for a
polynomial $f_0$. Since $\psi_\lambda$ only changes the
pair variables,
\[
 G=\psi_\lambda(f)
   =\Delta_{s,t}\psi_\lambda(f_0).
\]
Thus all residual $x$--$y$ factors also divide $G$.
The linear factors appearing in the three products are
pairwise nonassociate irreducibles in the full polynomial
ring in the block and residual variables. Unique
factorization combines their individual divisibilities,
with the cluster multiplicities, to give $D_\lambda\mid G$.
\end{proof}

\subsection{Symmetry, degree, and the Hilbert-series bound}

Put $\boldsymbol u^{(i)}=(u_1^{(i)},\ldots,u_{r_i}^{(i)})$
for $1\leq i\leq m$, and define
\begin{equation}\label{eq:block-symmetric-ring}
 \mathcal S_\lambda:=
 \C[\boldsymbol u^{(1)},\ldots,\boldsymbol u^{(m)},
       \boldsymbol x,\boldsymbol y,\boldsymbol z]
 ^{\left(\prod_{i=1}^m\mathfrak S_{r_i}\right)
    \times\mathfrak S_s\times\mathfrak S_t\times\mathfrak S_c}.
\end{equation}
This denotes separate symmetry within each block-size family
and within each color family of residual variables.

\begin{proposition}\label{prop:gordon-hilbert-bound}
The evaluation in \eqref{eq:gordon-image} gives a graded
injection
\begin{equation}\label{eq:divisible-image-injection}
 \mathcal Q_\lambda\lhook\joinrel\longrightarrow
 D_\lambda\mathcal S_\lambda.
\end{equation}
The polynomial $D_\lambda$ is homogeneous of degree
\begin{equation}\label{eq:degree}
 d_\lambda=
 \sum_{i,j=1}^m\min(i,j)r_ir_j-R+r_m(s+t+c)+st.
\end{equation}
Consequently, with respect to polynomial degree,
\begin{equation}\label{eq:gordon-hilbert-bound}
 H_{\mathcal Q_\lambda}(q)
 \leq
 \frac{q^{d_\lambda}}
 {\prod_{i=1}^m(q)_{r_i}(q)_s(q)_t(q)_c},
\end{equation}
where the inequality is coefficientwise.
\end{proposition}
\begin{proof}
First, $G=\psi_\lambda(f)$ is symmetric within each
block-size family: a permutation of equal-size blocks is
induced by a permutation of the original pair variables.
The symmetries among the residual variables are preserved by the evaluation.
Thus $G\in\mathcal S_\lambda$.

The divisor $D_\lambda$ has these same symmetries.
For blocks of size $i$, its same-size factor is the
$2i$-th power of a Vandermonde product, which is symmetric
because $2i$ is even. Permuting variables within any family
permutes the factors in the remaining products.
By Proposition~\ref{prop:gordon-divisor}, write $G=D_\lambda h$
in the full polynomial ring. For any permutation $\sigma$
in the group defining \eqref{eq:block-symmetric-ring},
\[
 D_\lambda\sigma(h)
 =\sigma(D_\lambda h)=\sigma(G)=G=D_\lambda h.
\]
The polynomial ring is an integral domain and $D_\lambda\ne0$,
so $\sigma(h)=h$. Hence $h\in\mathcal S_\lambda$.
Together with \eqref{eq:gordon-image}, this proves the
injection \eqref{eq:divisible-image-injection}.

To compute the degree, the same-size factors contribute
$\sum_i2i\binom{r_i}{2}$, and the different-size factors
contribute $\sum_{i<j}2i r_ir_j$. Therefore
\begin{align}\label{eq:cluster-degree-computation}
 \deg D_{\mathrm{cl},\lambda}
 &=\sum_{i=1}^m2i\binom{r_i}{2}
    +2\sum_{1\leq i<j\leq m}i r_ir_j\nonumber\\
 &=\sum_{i=1}^m i r_i^2
    +2\sum_{1\leq i<j\leq m}i r_ir_j
    -\sum_{i=1}^m i r_i\nonumber\\
 &=\sum_{i,j=1}^m\min(i,j)r_ir_j-R.
\end{align}
Each of the $r_m$ blocks of size $m$ contributes one
linear factor for each of the $s+t+c$ residual variables,
and $\Delta_{s,t}$ has $st$ linear factors. Adding these
degrees proves \eqref{eq:degree}.

For a family of $r$ variables of degree one, the fundamental
theorem of symmetric polynomials gives
\[
 \C[v_1,\ldots,v_r]^{\mathfrak S_r}
 =\C[e_1(\boldsymbol v),\ldots,e_r(\boldsymbol v)],
 \qquad \deg e_j(\boldsymbol v)=j,
\]
where the elementary symmetric polynomials are algebraically
independent. Its Hilbert series is consequently
\[
 \prod_{j=1}^r(1+q^j+q^{2j}+\cdots)
 =\frac{1}{\prod_{j=1}^r(1-q^j)}
 =\frac1{(q)_r}.
\]
The separately symmetric ring \eqref{eq:block-symmetric-ring}
is the tensor product of these rings for its disjoint
variable families. Hence
\begin{equation}\label{eq:block-symmetric-hilbert}
 H_{\mathcal S_\lambda}(q)
 =\frac1{\prod_{i=1}^m(q)_{r_i}(q)_s(q)_t(q)_c}.
\end{equation}
The factor for blocks of size $i$ is $1/(q)_{r_i}$ because
their variables have degree one; the block size $i$ does
not change that grading. An empty family contributes the
ring $\C$ and the Hilbert-series factor $1/(q)_0=1$.

Multiplication by the nonzero homogeneous polynomial
$D_\lambda$ shifts polynomial degree by $d_\lambda$.
More explicitly,
\[
 (D_\lambda\mathcal S_\lambda)[n]
 =D_\lambda\mathcal S_\lambda[n-d_\lambda],
\]
where a negative-degree component is zero. Taking degree
$n$ in \eqref{eq:divisible-image-injection} therefore gives
\begin{equation}\label{eq:gordon-degree-bound}
 \dim\mathcal Q_\lambda[n]
 \leq\dim\mathcal S_\lambda[n-d_\lambda]
 \qquad(n\geq0).
\end{equation}
Summing over $n$ and using
\eqref{eq:block-symmetric-hilbert} proves
\eqref{eq:gordon-hilbert-bound}.
\end{proof}

No claim that the image equals
$D_\lambda\mathcal S_\lambda$ is needed; further restrictions
on the image can only improve the coefficientwise upper bound.
Finally, the finite filtration \eqref{eq:gordon-filtration}
and the vanishing of the quotients with a part larger than
$m$ give
\begin{equation}\label{eq:gordon-summed-bound}
 H_U(q)
 =\sum_{\substack{\lambda\vdash R\\\text{all parts}\leq m}}
   H_{\mathcal Q_\lambda}(q)
 \leq
 \sum_{\substack{r_1,\ldots,r_m\geq0\\\sum_i i r_i=R}}
 \frac{q^{d_\lambda}}
 {\prod_{i=1}^m(q)_{r_i}(q)_s(q)_t(q)_c}.
\end{equation}
For $R=0$, the sum has the single term corresponding to
$r_1=\cdots=r_m=0$ and $D_\lambda=\Delta_{s,t}$.
In the next section we apply \eqref{eq:gordon-summed-bound}
to $U=\mathcal U_R$ and restore the conformal-weight shift
$a+b+c=2R+s+t+c$.

\section{The Hilbert-series bound}
\label{sec:hilbert-series-bound}

\subsection{Jet degree and conformal weight}

We first distinguish the grading used in the dual polynomial
spaces from the grading defining the character. The monomial
\[
 e_{i_1}\cdots e_{i_a}
 f_{j_1}\cdots f_{j_b}
 h_{k_1}\cdots h_{k_c}
\]
has jet degree
\[
 D=\sum_{\alpha=1}^a i_\alpha
   +\sum_{\beta=1}^b j_\beta
   +\sum_{\gamma=1}^c k_\gamma.
\]
Since each generator with index $i$ has conformal weight $i+1$,
its conformal weight is $D+a+b+c$. Under the duality of
Section~\ref{sec:dual-first-filtration}, jet degree corresponds
to ordinary polynomial degree. Thus define
\begin{equation}\label{eq:fixed-multidegree-hilbert}
 H^{\mathrm{jet}}_{a,b,c}(q)
 :=\sum_{D\geq0}\dim C_{m;a,b,c}[D]q^D
 =\sum_{D\geq0}\dim W_{a,b,c}[D]q^D.
\end{equation}
The conformal Hilbert series is then
\begin{equation}\label{eq:conformal-jet-shift}
 H_{C_m}(q)
 =\sum_{a,b,c\geq0}q^{a+b+c}
   H^{\mathrm{jet}}_{a,b,c}(q).
\end{equation}
In what follows, $H^{\mathrm{jet}}_E(q)$ denotes the Hilbert
series of a graded polynomial space or its subquotient $E$
with respect to ordinary polynomial degree.

For fixed $(a,b,c)$, put $M=\min(a,b)$. By
\eqref{eq:first-filtration-dimensions},
\begin{equation}\label{eq:first-hilbert-additivity}
 H^{\mathrm{jet}}_{a,b,c}(q)
 =\sum_{R=0}^M H^{\mathrm{jet}}_{\mathcal U_R}(q).
\end{equation}
This is an equality because the first filtration is finite
and graded, and dimensions add in each homogeneous component.
For each $R$, set
\[
 s=a-R,\qquad t=b-R.
\]
Each of the $R$ pair variables replaces one original
$x$-variable and one original $y$-variable. Consequently,
the conformal shift for every quotient of this fixed
multidegree remains
\begin{equation}\label{eq:pair-conformal-shift}
 a+b+c=2R+s+t+c.
\end{equation}
The specializations preserve polynomial degree; they do not
replace this shift by the number of variables after
specialization.

\subsection{Hilbert-series bound of $C_m$}

Apply the Gordon filtration to $U=\mathcal U_R$, and denote
its quotient \eqref{eq:gordon-quotient} by
$\mathcal Q_{R,\lambda}$ to indicate the dependence on $R$.
The dependence on the fixed multidegree $(a,b,c)$ is suppressed.
The finite filtration and Proposition~\ref{prop:gordon-hilbert-bound}
give
\begin{align}\label{eq:paired-hilbert-bound}
 H^{\mathrm{jet}}_{\mathcal U_R}(q)
 &=\sum_{\substack{\lambda\vdash R\\\text{all parts}\leq m}}
   H^{\mathrm{jet}}_{\mathcal Q_{R,\lambda}}(q)\nonumber\\
 &\leq
 \sum_{\substack{r_1,\ldots,r_m\geq0\\\sum_i i r_i=R}}
 \frac{q^{d_\lambda}}
 {\prod_{i=1}^m(q)_{r_i}(q)_s(q)_t(q)_c},
\end{align}
where $\lambda=(m^{r_m},\ldots,1^{r_1})$ and
\[
 d_\lambda
 =\sum_{i,j=1}^m\min(i,j)r_ir_j
  -R+r_m(s+t+c)+st.
\]
In particular, combining the two filtrations yields the
fixed-multidegree bound
\begin{equation}\label{eq:fixed-multidegree-bound}
 H^{\mathrm{jet}}_{a,b,c}(q)
 \leq
 \sum_{R=0}^{\min(a,b)}
 \ \sum_{\substack{r_1,\ldots,r_m\geq0\\\sum_i i r_i=R}}
 \frac{q^{d_\lambda}}
 {\prod_{i=1}^m(q)_{r_i}(q)_{a-R}(q)_{b-R}(q)_c}.
\end{equation}
After restoring the conformal shift in
\eqref{eq:pair-conformal-shift}, the exponent attached to
this quotient is
\begin{align}\label{eq:conformal-exponent-computation}
 2R+s+t+c+d_\lambda
 &=2R+s+t+c+\sum_{i,j=1}^m\min(i,j)r_ir_j
   -R+r_m(s+t+c)+st\nonumber\\
 &=\sum_{i,j=1}^m\min(i,j)r_ir_j
   +R+(r_m+1)(s+t+c)+st\nonumber\\
 &=\sum_{i,j=1}^m\min(i,j)r_ir_j
   +\sum_{i=1}^m i r_i+(r_m+1)(s+t+c)+st.
\end{align}

To sum over all multidegrees, we use the bijection between the
index sets
\[
 \begin{gathered}
 a,b,c\geq0,\quad 0\leq R\leq\min(a,b),\quad
 \lambda\vdash R,\quad\text{all parts of $\lambda$}\leq m,
 \\
 \longleftrightarrow
 \\
 r_1,\ldots,r_m,s,t,c\geq0.
 \end{gathered}
\]
In the forward direction, $r_i$ is the multiplicity of the
part $i$ in $\lambda$, and $s=a-R$, $t=b-R$. Conversely,
a tuple in the second index set determines
\[
 R=\sum_{i=1}^m i r_i,\qquad
 \lambda=(m^{r_m},\ldots,1^{r_1}),\qquad
 a=R+s,\qquad b=R+t.
\]
The inequality $R\leq\min(a,b)$ is then automatic.
Consequently, \eqref{eq:conformal-jet-shift},
\eqref{eq:fixed-multidegree-bound}, and
\eqref{eq:conformal-exponent-computation} give
\begin{equation}\label{eq:Sm}
 H_{C_m}(q)\leq\mathcal S_m(q):=
 \sum_{r_1,\ldots,r_m,s,t,c\geq0}
 \frac{
 q^{\sum_{i,j=1}^m\min(i,j)r_ir_j+
       \sum_{i=1}^m i r_i+(r_m+1)(s+t+c)+st}}
 {\prod_{i=1}^m(q)_{r_i}(q)_s(q)_t(q)_c}.
\end{equation}

These sums are well defined in $\C[[q]]$. Indeed,
$d_\lambda$ is the degree of the polynomial $D_\lambda$,
so it is nonnegative. The numerator exponent in
\eqref{eq:Sm} therefore satisfies
\[
 2R+s+t+c+d_\lambda\geq2R+s+t+c.
\]
Every reciprocal factor $1/(q)_j$ has an expansion in
nonnegative powers of $q$. Hence a contribution to the
coefficient of $q^N$ requires $2R+s+t+c\leq N$.
There are only finitely many possible $R,s,t,c$, and,
for each $R$, only finitely many tuples satisfying
$\sum_i i r_i=R$. Thus each coefficient involves finitely
many summands, which also justifies the preceding change
of indices.

\subsection{Cumulative multiplicities}

Introduce
\begin{equation}\label{eq:cumulative-multiplicities}
 n_j:=r_j+r_{j+1}+\cdots+r_m\quad(1\leq j\leq m),
 \qquad n_{m+1}:=0.
\end{equation}
Then $r_j=n_j-n_{j+1}$, so this is a bijection between
nonnegative multiplicity tuples and chains
\[
 n_1\geq n_2\geq\cdots\geq n_m\geq0.
\]
For the quadratic term, the identity
\[
 \min(i,j)
 =\sum_{\ell=1}^m
   \mathbf1_{\ell\leq i}\mathbf1_{\ell\leq j}
\]
gives
\begin{align}\label{eq:cumulative-quadratic-identity}
 \sum_{i,j=1}^m\min(i,j)r_ir_j
 &=\sum_{\ell=1}^m
   \left(\sum_{i=\ell}^m r_i\right)
   \left(\sum_{j=\ell}^m r_j\right)\nonumber\\
 &=\sum_{\ell=1}^m n_\ell^2.
\end{align}
Similarly,
\begin{equation}\label{eq:cumulative-linear-identity}
 \sum_{j=1}^m n_j
 =\sum_{j=1}^m\sum_{i=j}^m r_i
 =\sum_{i=1}^m i r_i=R,
 \qquad r_m=n_m.
\end{equation}
The denominator transforms according to
\[
 \prod_{i=1}^m(q)_{r_i}
 =\left(\prod_{j=1}^{m-1}(q)_{n_j-n_{j+1}}\right)(q)_{n_m}.
\]
Substituting these identities into \eqref{eq:Sm} gives
\begin{equation}\label{eq:chainpositive}
 \mathcal S_m(q)=
 \sum_{\substack{n_1\geq\cdots\geq n_m\geq0\\s,t,c\geq0}}
 \frac{q^{\sum_{j=1}^m n_j(n_j+1)+(n_m+1)(s+t+c)+st}}
 {\left(\prod_{j=1}^{m-1}(q)_{n_j-n_{j+1}}\right)
  (q)_{n_m}(q)_s(q)_t(q)_c}.
\end{equation}

\subsection{Summing the residual variables}

We regard $q$ as a formal variable. For $N\in\N$, the finite
$q$-Pochhammer symbol is
\[
 (a;q)_N:=\prod_{j=0}^{N-1}(1-aq^j),\qquad (a;q)_0:=1,
\]
and we write
\[
 (a;q)_\infty:=\prod_{j=0}^{\infty}(1-aq^j),\qquad
 (q)_N=\prod_{j=1}^{N}(1-q^j)=(q;q)_N,\qquad
 (q)_\infty=(q;q)_\infty.
\]
The infinite products are interpreted in the corresponding formal
power-series rings. In particular, $(q)_N$ and $(q)_\infty$ have
constant term $1$ and are invertible in $\C[[q]]$. For every
integer $A\geq1$,
\begin{equation}\label{eq:pochhammer-tail}
 (q^A;q)_\infty=\frac{(q)_\infty}{(q)_{A-1}}.
\end{equation}
Limits in this section are taken coefficientwise: for every fixed $D$, the
coefficient of $q^D$ is eventually constant and equal to the
coefficient of the asserted limit. Equivalently, each truncation
modulo $q^{D+1}$ eventually agrees. An infinite sum is rearranged
only when finitely many summands contribute to each coefficient.

We shall use Euler's identity and the $q$-binomial theorem:
\begin{equation}\label{eq:euler-qbinomial}
 \sum_{j\geq0}\frac{z^j}{(q)_j}
 =\frac1{(z;q)_\infty},\qquad
 \sum_{j\geq0}\frac{(a;q)_j}{(q)_j}z^j
 =\frac{(az;q)_\infty}{(z;q)_\infty}.
\end{equation}
These identities hold as formal series in $q$ and $z$, with
polynomial coefficients in $a$. For completeness, if
$B(z)=\sum_{j\geq0}(a;q)_jz^j/(q)_j$, then
\[
 (1-z)B(z)=(1-az)B(qz),\qquad B(0)=1,
\]
because the coefficients $b_j=(a;q)_j/(q)_j$ satisfy
\[
 (1-q^j)b_j=(1-aq^{j-1})b_{j-1}\qquad(j\geq1).
\]
The quotient of infinite products in the second identity of
\eqref{eq:euler-qbinomial} satisfies the same functional equation
and constant-term condition. The recurrence uniquely determines
its coefficients, proving the identity. Taking $a=0$ gives the
first identity. In applying \eqref{eq:euler-qbinomial} below, we set
$z=q^A$ with $A\geq1$, so the substitutions define elements of
$\C[[q]]$.

Fix $n\geq0$, and consider
\[
 \mathcal R_n(q):=
 \sum_{s,t,c\geq0}
 \frac{q^{(n+1)(s+t+c)+st}}
 {(q)_n(q)_s(q)_t(q)_c}.
\]
For fixed $s,t$, Euler's identity gives
\[
 \sum_{c\geq0}\frac{q^{(n+1)c}}{(q)_c}
 =\frac1{(q^{n+1};q)_\infty}
 =\frac{(q)_n}{(q)_\infty}.
\]
This factor cancels the denominator $(q)_n$ in $\mathcal R_n(q)$.
Consequently,
\begin{align*}
 \mathcal R_n(q)
 &=\frac1{(q)_\infty}
   \sum_{s,t\geq0}
   \frac{q^{(n+1)(s+t)+st}}{(q)_s(q)_t}\\
 &=\frac1{(q)_\infty}
   \sum_{s\geq0}\frac{q^{(n+1)s}}{(q)_s}
   \left(\sum_{t\geq0}
         \frac{q^{(n+s+1)t}}{(q)_t}\right).
\end{align*}
For each fixed $s$, the remaining inner sum is
\[
 \sum_{t\geq0}\frac{q^{(n+s+1)t}}{(q)_t}
 =\frac1{(q^{n+s+1};q)_\infty}
 =\frac{(q)_{n+s}}{(q)_\infty}.
\]
Using
\[
 (q)_{n+s}=(q)_n(q^{n+1};q)_s
\]
and the second identity of \eqref{eq:euler-qbinomial}, with
$a=z=q^{n+1}$, we obtain
\begin{align*}
 \sum_{s\geq0}\frac{(q)_{n+s}}{(q)_s}q^{(n+1)s}
 &=(q)_n\sum_{s\geq0}
   \frac{(q^{n+1};q)_s}{(q)_s}(q^{n+1})^s\\
 &=(q)_n\frac{(q^{2n+2};q)_\infty}
                   {(q^{n+1};q)_\infty}\\
 &=(q)_n
   \frac{(q)_\infty/(q)_{2n+1}}
        {(q)_\infty/(q)_n}
  =\frac{(q)_n^2}{(q)_{2n+1}}.
\end{align*}
Thus
\begin{align}\label{eq:residualsum}
 &\sum_{s,t,c\geq0}
 \frac{q^{(n+1)(s+t+c)+st}}
 {(q)_n(q)_s(q)_t(q)_c}\nonumber\\
 &\qquad=
 \frac1{(q)_\infty^2}
 \sum_{s\geq0}\frac{(q)_{n+s}}{(q)_s}q^{(n+1)s}
 =\frac{(q)_n^2}{(q)_\infty^2(q)_{2n+1}}.
\end{align}
These computations include $n=0$: the middle sum is then
$\sum_{s\geq0}q^s=1/(1-q)$, in agreement with the last expression.
The exponent $(n+1)(s+t+c)+st$ is at least $s+t+c$, so every
coefficient receives contributions from only finitely many triples.
This justifies the successive summations formally.

Applying \eqref{eq:residualsum} with $n=n_m$ to
\eqref{eq:chainpositive} gives
\begin{equation}\label{eq:chainsum}
 \begin{split}
 \mathcal S_m(q)&=\frac{\mathcal T_m(q)}{(q)_\infty^2},\\
 \mathcal T_m(q)&:=
 \sum_{n_1\geq\cdots\geq n_m\geq0}
 \frac{q^{\sum_{j=1}^m n_j(n_j+1)}(q)_{n_m}^2}
 {\left(\prod_{j=1}^{m-1}(q)_{n_j-n_{j+1}}\right)
  (q)_{2n_m+1}}.
 \end{split}
\end{equation}
The product over $1\leq j\leq m-1$ is $1$ when $m=1$.

\subsection{Some Combinatorial identities}

For $N\in\N$ and $j\in\mathbb Z$, define the Gaussian
$q$-binomial coefficient by
\begin{equation}\label{eq:gaussian-definition}
 \qbinom Nj:=
 \begin{cases}
 \displaystyle\frac{(q)_N}{(q)_j(q)_{N-j}},&0\leq j\leq N,\\[6pt]
 0,&j<0\text{ or }j>N.
 \end{cases}
\end{equation}
In particular, $\qbinom N0=\qbinom NN=1$ and
$\qbinom Nj=\qbinom N{N-j}$. The finite $q$-binomial theorem is
\begin{equation}\label{eq:finite-qbinomial}
 (z;q)_N=
 \sum_{j=0}^{N}(-1)^j q^{j(j-1)/2}\qbinom Nj z^j.
\end{equation}

\begin{lemma}\label{lem:finite}
For every $n\geq0$,
\begin{equation}\label{eq:finite}
 (q)_n^2=
 \sum_{r=0}^n(-1)^r(2r+1)q^{r(r+1)/2}
 \qbinom{2n+1}{n-r}.
\end{equation}
Consequently,
\begin{equation}\label{eq:Jacobi}
 \sum_{r\geq0}(-1)^r(2r+1)q^{r(r+1)/2}=(q)_\infty^3.
\end{equation}
\end{lemma}
\begin{proof}
Set $P_n(z)=(z;q)_{2n+1}=\prod_{j=0}^{2n}(1-zq^j)$.
Apply $z\,d/dz$ to \eqref{eq:finite-qbinomial} with $N=2n+1$,
and then set $z=q^{-n}$. On the product side,
\[
 zP_n'(z)=
 \sum_{j=0}^{2n}(-zq^j)
 \prod_{\substack{0\leq\ell\leq2n\\\ell\neq j}}
 (1-zq^\ell).
\]
At $z=q^{-n}$, every term except $j=n$ contains the factor
$1-q^{-n}q^n=0$. Therefore
\begin{align}\label{eq:derivative-product-evaluation}
 \left.zP_n'(z)\right|_{z=q^{-n}}
 &=-\prod_{\substack{0\leq\ell\leq2n\\\ell\neq n}}
       (1-q^{\ell-n})\nonumber\\
 &=-\prod_{a=1}^{n}(1-q^{-a})(1-q^a)\nonumber\\
 &=(-1)^{n+1}q^{-n(n+1)/2}(q)_n^2.
\end{align}
In the last equality we used $1-q^{-a}=-q^{-a}(1-q^a)$.

Differentiating the sum side gives
\begin{equation}\label{eq:derivative-sum-evaluation}
 \left.zP_n'(z)\right|_{z=q^{-n}}
 =\sum_{j=0}^{2n+1}
 j(-1)^j q^{j(j-1)/2-nj}\qbinom{2n+1}j.
\end{equation}
Pair the indices
\[
 j_-=n-r,\qquad j_+=n+r+1\qquad(0\leq r\leq n).
\]
These pairs partition the indices $0,\ldots,2n+1$. Since
$j_-+j_+=2n+1$, symmetry of the Gaussian coefficients gives
\[
 \qbinom{2n+1}{j_-}=
 \qbinom{2n+1}{j_+}=
 \qbinom{2n+1}{n-r}.
\]
The two powers of $q$ agree, because
\begin{equation}\label{eq:paired-exponents}
 \frac{j_\pm(j_\pm-1)}2-nj_\pm
 =-\frac{n(n+1)}2+\frac{r(r+1)}2.
\end{equation}
Moreover,
\begin{align*}
 j_-(-1)^{j_-}+j_+(-1)^{j_+}
 &=(-1)^{n-r}\bigl((n-r)-(n+r+1)\bigr)\\
 &=(-1)^{n+r+1}(2r+1).
\end{align*}
Thus the contribution of a pair to
\eqref{eq:derivative-sum-evaluation} equals
\[
 (-1)^{n+1}q^{-n(n+1)/2}
 (-1)^r(2r+1)q^{r(r+1)/2}
 \qbinom{2n+1}{n-r}.
\]
Comparing with \eqref{eq:derivative-product-evaluation} and
cancelling $(-1)^{n+1}q^{-n(n+1)/2}$ proves
\eqref{eq:finite}. The substitution $z=q^{-n}$ was made in
$\C(q)$; after this cancellation, both sides of
\eqref{eq:finite} belong to $\C[[q]]$.

For each fixed $r$, as $n\to\infty$ with $n\geq r$,
\begin{equation}\label{eq:gaussian-central-limit}
 \qbinom{2n+1}{n-r}
 =\frac{(q)_{2n+1}}{(q)_{n-r}(q)_{n+r+1}}
 \longrightarrow\frac1{(q)_\infty}
\end{equation}
coefficientwise. Indeed, for every fixed $D$, each finite product
in the quotient agrees with $(q)_\infty$ modulo $q^{D+1}$ once
its index is at least $D$, and inversion preserves this agreement.
In the coefficient of $q^D$ on the right-hand side of
\eqref{eq:finite}, only indices satisfying $r(r+1)/2\leq D$
can occur. This is a finite set independent of sufficiently large
$n$, since the Gaussian coefficients have no negative powers of
$q$. We may therefore pass to the coefficientwise limit to obtain
\[
 (q)_\infty^2=
 \frac1{(q)_\infty}
 \sum_{r\geq0}(-1)^r(2r+1)q^{r(r+1)/2}.
\]
Multiplication by $(q)_\infty$ proves \eqref{eq:Jacobi}.
\end{proof}

\begin{lemma}\label{lem:kernel}
For integers $N,d\geq0$,
\begin{equation}\label{eq:kernelbasic}
 \sum_{j=0}^N
 \frac{q^{j(j+d)}}{(q)_{N-j}(q)_j(q)_{j+d}}
 =\frac1{(q)_N(q)_{N+d}}.
\end{equation}
In particular, for $n\geq r\geq0$,
\begin{equation}\label{eq:kernel}
 \sum_{j=r}^n
 \frac{q^{j(j+1)}}{(q)_{n-j}(q)_{j-r}(q)_{j+r+1}}
 =\frac{q^{r(r+1)}}{(q)_{n-r}(q)_{n+r+1}},
\end{equation}
and, for $r\geq0$,
\begin{equation}\label{eq:infinitekernel}
 \sum_{j=r}^\infty
 \frac{q^{j(j+1)}}{(q)_{j-r}(q)_{j+r+1}}
 =\frac{q^{r(r+1)}}{(q)_\infty}.
\end{equation}
\end{lemma}
\begin{proof}
For $N\geq1$, the Gaussian coefficients satisfy
\begin{equation}\label{eq:gaussian-recurrence}
 \qbinom Nj=
 \qbinom{N-1}j+q^{N-j}\qbinom{N-1}{j-1}
 \qquad(0\leq j\leq N).
\end{equation}

Multiplying the left-hand side of \eqref{eq:kernelbasic} by
$(q)_N$ gives
\begin{equation}\label{eq:FN-definition}
 F_N(d):=\sum_{j=0}^N\qbinom Nj
                 \frac{q^{j(j+d)}}{(q)_{j+d}}.
\end{equation}
We prove simultaneously for all $d\geq0$, by induction on $N$,
that
\begin{equation}\label{eq:FN-evaluation}
 F_N(d)=\frac1{(q)_{N+d}}.
\end{equation}
For $N=0$, the only term of \eqref{eq:FN-definition} has $j=0$,
and $F_0(d)=1/(q)_d$ for every $d\geq0$.

Suppose $N\geq1$. Applying \eqref{eq:gaussian-recurrence} yields
\begin{align*}
 F_N(d)
 &=\sum_{j=0}^{N-1}\qbinom{N-1}j
          \frac{q^{j(j+d)}}{(q)_{j+d}}\\
 &\quad+\sum_{j=1}^{N}\qbinom{N-1}{j-1}
          \frac{q^{N-j+j(j+d)}}{(q)_{j+d}}.
\end{align*}
The first sum is $F_{N-1}(d)$. In the second sum, put
$j=\ell+1$. Then
\[
 \begin{aligned}
 N-j+j(j+d)
 &=N-\ell-1+(\ell+1)(\ell+d+1)\\
 &=N+d+\ell(\ell+d+1),
 \end{aligned}
\]
and $(q)_{j+d}=(q)_{\ell+(d+1)}$. Consequently,
\begin{align}\label{eq:FN-recurrence}
 F_N(d)
 &=F_{N-1}(d)
   +q^{N+d}\sum_{\ell=0}^{N-1}\qbinom{N-1}\ell
     \frac{q^{\ell(\ell+d+1)}}{(q)_{\ell+d+1}}\nonumber\\
 &=F_{N-1}(d)+q^{N+d}F_{N-1}(d+1).
\end{align}
The induction hypothesis applies at both $d$ and $d+1$, so
\begin{align*}
 F_N(d)
 &=\frac1{(q)_{N+d-1}}
   +\frac{q^{N+d}}{(q)_{N+d}}\\
 &=\frac{1-q^{N+d}+q^{N+d}}{(q)_{N+d}}
  =\frac1{(q)_{N+d}}.
\end{align*}
Here $N+d\geq1$, so the index $N+d-1$ is nonnegative.
This proves \eqref{eq:FN-evaluation}. Dividing by $(q)_N$
proves \eqref{eq:kernelbasic}.

To obtain \eqref{eq:kernel}, set
\[
 N=n-r,\qquad d=2r+1,\qquad j=r+\ell.
\]
The exponent and denominator indices become
\[
 \begin{aligned}
 j(j+1)&=r(r+1)+\ell(\ell+2r+1),\\
 n-j&=N-\ell,\qquad j-r=\ell,\qquad
 j+r+1=\ell+d.
 \end{aligned}
\]
Thus \eqref{eq:kernelbasic} gives
\begin{align*}
 &\sum_{j=r}^n
 \frac{q^{j(j+1)}}{(q)_{n-j}(q)_{j-r}(q)_{j+r+1}}\\
 &\qquad=
 q^{r(r+1)}\sum_{\ell=0}^{N}
 \frac{q^{\ell(\ell+d)}}{(q)_{N-\ell}(q)_\ell(q)_{\ell+d}}\\
 &\qquad=
 \frac{q^{r(r+1)}}{(q)_N(q)_{N+d}}
 =\frac{q^{r(r+1)}}{(q)_{n-r}(q)_{n+r+1}}.
\end{align*}

Finally, fix $d\geq0$ and let $N\to\infty$ in
\eqref{eq:FN-evaluation}. For every fixed $j$,
\[
 \qbinom Nj=
 \frac{1}{(q)_j}\prod_{a=N-j+1}^{N}(1-q^a)
 \longrightarrow\frac1{(q)_j}
\]
coefficientwise, where the product is empty if $j=0$.
At degree $D$, only $j$ with $j(j+d)\leq D$ contribute to
\eqref{eq:FN-definition}. Hence the limit can be taken termwise,
and it gives
\begin{equation}\label{eq:kernel-limit-rectangle}
 \sum_{j\geq0}\frac{q^{j(j+d)}}{(q)_j(q)_{j+d}}
 =\frac1{(q)_\infty}.
\end{equation}
Applying \eqref{eq:kernel-limit-rectangle} with $d=2r+1$
and shifting $j=r+\ell$ as above yields
\[
 \begin{aligned}
 \sum_{j=r}^\infty
 \frac{q^{j(j+1)}}{(q)_{j-r}(q)_{j+r+1}}
 &=q^{r(r+1)}
   \sum_{\ell\geq0}
   \frac{q^{\ell(\ell+2r+1)}}{(q)_\ell(q)_{\ell+2r+1}}\\
 &=\frac{q^{r(r+1)}}{(q)_\infty},
 \end{aligned}
\]
which is \eqref{eq:infinitekernel}.
\end{proof}

\subsection{The chain summation}
\label{subsec:chain-summation}
In this section, we will show that
\begin{equation}\label{eq:Sidentity}
 \mathcal S_m(q)=\frac{(q^{2m+1};q^{2m+1})_\infty^3}
 {(q)_\infty^3}.
\end{equation}

Set
\begin{equation}\label{eq:chain-alpha}
 \alpha_r:=(-1)^r(2r+1)q^{r(r+1)/2}
 \qquad(r\geq0).
\end{equation}
For integers $\ell\geq1$ and $r\geq0$, define
\begin{equation}\label{eq:chain-kernel-definition}
 G_\ell(r):=
 \sum_{n_1\geq\cdots\geq n_\ell\geq r}
 \frac{q^{\sum_{i=1}^\ell n_i(n_i+1)}}
 {\left(\prod_{i=1}^{\ell-1}(q)_{n_i-n_{i+1}}\right)
  (q)_{n_\ell-r}(q)_{n_\ell+r+1}}.
\end{equation}
For $\ell=1$, the product of consecutive-difference factors
is empty and is equal to $1$.

\begin{lemma}\label{lem:chain-summation}
For every $\ell\geq1$ and $r\geq0$,
\begin{equation}\label{eq:chain-kernel-evaluation}
 G_\ell(r)=\frac{q^{\ell r(r+1)}}{(q)_\infty}.
\end{equation}
Moreover,
\begin{equation}\label{eq:chain-evaluation}
 \mathcal T_m(q)
 =\sum_{r\geq0}\alpha_rG_m(r)
 =\frac{(q^{2m+1};q^{2m+1})_\infty^3}{(q)_\infty}.
\end{equation}
In particular, \eqref{eq:Sidentity} holds.
\end{lemma}
\begin{proof}
By the definition of the Gaussian coefficient,
\[
 \frac{1}{(q)_{2n+1}}\qbinom{2n+1}{n-r}
 =\frac1{(q)_{n-r}(q)_{n+r+1}}
 \qquad(0\leq r\leq n).
\]
Dividing \eqref{eq:finite} by $(q)_{2n+1}$ therefore gives
\begin{equation}\label{eq:finite-ratio-expansion}
 \frac{(q)_n^2}{(q)_{2n+1}}
 =\sum_{r=0}^n
 \frac{\alpha_r}{(q)_{n-r}(q)_{n+r+1}}.
\end{equation}
Substitute this identity with $n=n_m$ into
\eqref{eq:chainsum}. The resulting expression is
\begin{align}\label{eq:chain-reindexing}
 \mathcal T_m(q)
 &=\sum_{n_1\geq\cdots\geq n_m\geq0}\sum_{r=0}^{n_m}
 \frac{\alpha_r q^{\sum_{i=1}^m n_i(n_i+1)}}
 {\left(\prod_{i=1}^{m-1}(q)_{n_i-n_{i+1}}\right)
  (q)_{n_m-r}(q)_{n_m+r+1}}\nonumber\\
 &=\sum_{r\geq0}\alpha_rG_m(r).
\end{align}

It remains to evaluate $G_\ell(r)$. Fix $r$, and suppose first
that $\ell\geq2$. For fixed $n_1,\ldots,n_{\ell-1}$, the
innermost sum in \eqref{eq:chain-kernel-definition} is
\[
 \sum_{j=r}^{n_{\ell-1}}
 \frac{q^{j(j+1)}}
 {(q)_{n_{\ell-1}-j}(q)_{j-r}(q)_{j+r+1}}
 =\frac{q^{r(r+1)}}
 {(q)_{n_{\ell-1}-r}(q)_{n_{\ell-1}+r+1}},
\]
by \eqref{eq:kernel} with upper index $n_{\ell-1}$.
Substitution into the outer sum gives
\begin{align}\label{eq:chain-kernel-recurrence}
 G_\ell(r)
 &=q^{r(r+1)}
 \sum_{n_1\geq\cdots\geq n_{\ell-1}\geq r}
 \frac{q^{\sum_{i=1}^{\ell-1}n_i(n_i+1)}}
 {\left(\prod_{i=1}^{\ell-2}(q)_{n_i-n_{i+1}}\right)
  (q)_{n_{\ell-1}-r}(q)_{n_{\ell-1}+r+1}}\nonumber\\
 &=q^{r(r+1)}G_{\ell-1}(r).
\end{align}
In particular, the two terminal denominator factors have
the same form as before, with $n_\ell$ replaced by
$n_{\ell-1}$. This permits the next application of
\eqref{eq:kernel}.

After $\ell-1$ such finite summations, only $G_1(r)$ remains.
The infinite identity \eqref{eq:infinitekernel} yields
\[
 G_1(r)=\sum_{n_1=r}^\infty
 \frac{q^{n_1(n_1+1)}}{(q)_{n_1-r}(q)_{n_1+r+1}}
 =\frac{q^{r(r+1)}}{(q)_\infty}.
\]
Consequently,
\[
 G_\ell(r)
 =q^{(\ell-1)r(r+1)}G_1(r)
 =\frac{q^{\ell r(r+1)}}{(q)_\infty},
\]
which proves \eqref{eq:chain-kernel-evaluation}.

Combining \eqref{eq:chain-reindexing} with
\eqref{eq:chain-kernel-evaluation}, we obtain
\begin{align*}
 \mathcal T_m(q)
 &=\frac1{(q)_\infty}
   \sum_{r\geq0}\alpha_r q^{m r(r+1)}\\
 &=\frac1{(q)_\infty}
   \sum_{r\geq0}(-1)^r(2r+1)
   q^{r(r+1)/2+m r(r+1)}\\
 &=\frac1{(q)_\infty}
   \sum_{r\geq0}(-1)^r(2r+1)
   q^{(2m+1)r(r+1)/2}.
\end{align*}
Apply \eqref{eq:Jacobi} with $q$ replaced by
$q^{2m+1}$. This gives
\[
 \sum_{r\geq0}(-1)^r(2r+1)q^{(2m+1)r(r+1)/2}
 =(q^{2m+1};q^{2m+1})_\infty^3
\]
and proves \eqref{eq:chain-evaluation}. Finally,
\eqref{eq:chainsum} gives
\[
 \mathcal S_m(q)
 =\frac{\mathcal T_m(q)}{(q)_\infty^2}
 =\frac{(q^{2m+1};q^{2m+1})_\infty^3}{(q)_\infty^3},
\]
as asserted in \eqref{eq:Sidentity}.
\end{proof}

\section{The vacuum character and classical freeness}
\label{sec:vacuum-character}

We first derive the unshifted vacuum character from the
boundary-admissible affine character formula. This also specifies the
normalization used in the Hilbert-series comparison.

\begin{lemma}\label{lem:vacuum-character}
Let $p=2m+1\geq3$, $k=-2+2/p$, and $V=L_k(\slTwo)$. Then
\begin{equation}\label{eq:vacuum-character-product}
 \chi_p(q):=\operatorname{tr}_V q^{L_0}
 =\frac{(q^p;q^p)_\infty^3}{(q)_\infty^3}.
\end{equation}
\end{lemma}
\begin{proof}
Initially let $\tau$ lie in the upper half-plane and put
$q=e^{2\pi i\tau}$. The Sugawara central charge is
\begin{equation}\label{eq:vacuum-central-charge}
 c=\frac{3k}{k+2}
 =\frac{3(-2+2/p)}{2/p}=3(1-p).
\end{equation}
The normalized ordinary character is therefore
\[
 \operatorname{ch}^{\mathrm{norm}}_V(\tau)
 :=\operatorname{tr}_V q^{L_0-c/24}
 =q^{-c/24}\chi_p(q).
\]

In \cite[Example~1, equation~(8)]{KW}, set $u=p$, $j=0$, and
$t=0$. The weight indexed by $j=0$ is the vacuum weight. Their
formula for the normalized affine character gives
\begin{equation}\label{eq:vacuum-theta-quotient}
 \operatorname{ch}^{\mathrm{norm}}_V(\tau)
 =\lim_{z\to0}
 \frac{\vartheta_{11}(p\tau,z)}{\vartheta_{11}(\tau,z)},
\end{equation}
where $z$ is the Cartan variable. Specializing this variable to zero
forgets the finite-dimensional Cartan weights. The boundary-admissible
characters are also given in the unshifted convention in
\cite[Section~2.1.2]{AM}.

To evaluate the limit explicitly, put
\[
 \eta(\tau)=q^{1/24}(q)_\infty.
\]
Recall  the odd Jacobi theta function determined by
the product
\begin{equation}\label{eq:odd-theta-product}
 \vartheta_{11}(\tau,z)
 =2q^{1/8}\sin(\pi z)
  \prod_{n\geq1}(1-q^n)
  \bigl(1-2\cos(2\pi z)q^n+q^{2n}\bigr).
\end{equation}
Changing the sign of the odd theta function leaves the quotient in
\eqref{eq:vacuum-theta-quotient} unchanged. At $z=0$ we have
\[
 1-2\cos(2\pi z)q^n+q^{2n}=(1-q^n)^2,
\]
and differentiation of \eqref{eq:odd-theta-product} gives
\[
 \left.\frac{\partial}{\partial z}
       \vartheta_{11}(\tau,z)\right|_{z=0}
 =2\pi q^{1/8}(q)_\infty^3
 =2\pi\eta(\tau)^3.
\]
Thus both theta functions in \eqref{eq:vacuum-theta-quotient}
have a simple zero at $z=0$, and
\begin{equation}\label{eq:vacuum-normalized-eta}
 \operatorname{ch}^{\mathrm{norm}}_V(\tau)
 =\frac{\eta(p\tau)^3}{\eta(\tau)^3}.
\end{equation}

Finally,
\[
 \frac{\eta(p\tau)^3}{\eta(\tau)^3}
 =q^{(p-1)/8}
  \frac{(q^p;q^p)_\infty^3}{(q)_\infty^3}.
\]
Since $c/24=(1-p)/8$, removing the character normalization yields
\begin{align*}
 \chi_p(q)
 &=q^{c/24}\operatorname{ch}^{\mathrm{norm}}_V(\tau)\\
 &
 =\frac{(q^p;q^p)_\infty^3}{(q)_\infty^3}.
\end{align*}
Equality of these convergent expansions for $|q|<1$ gives the stated
identity of formal power series.
\end{proof}

\begin{proof}[Proof of Theorem \ref{thm:byproduct}]
The proof is clear by   (\ref{eq:chainsum}),  (\ref{eq:Sidentity}) and Lemma \ref{lem:vacuum-character}.
\end{proof}

\begin{proof}[Proof of Theorem~\ref{thm:main}]
By Lemma~\ref{lem:vacuum-character}, the Hilbert series of $V$ is
$\chi_{2m+1}(q)$. The Li filtration preserves conformal weight,
and taking its associated graded does not change the dimension of
any conformal-weight component. Hence
\[
 H_{\gr^F V}(q)=H_V(q)=\chi_{2m+1}(q).
\]
Combining \eqref{eq:character-lower-bound}, \eqref{eq:degeneration},
\eqref{eq:Sm}, and \eqref{eq:Sidentity}, we obtain the coefficientwise
inequalities
\[
 \chi_{2m+1}(q) \leq H_{J_\infty R_V}(q) \leq H_{B_m}(q)\leq H_{C_m}(q)
 \leq\mathcal S_m(q)=\chi_{2m+1}(q).
\]
Thus all inequalities are equalities. For each conformal weight
$N\geq0$, the canonical map induces a surjection
\[
 \pi_N:B_m[N]\twoheadrightarrow(\gr^F V)[N]
\]
between finite-dimensional vector spaces of equal dimension.
Thus $\ker\pi_N=0$ for each $N$, and hence $\ker\pi=0$.
Both two maps in  \eqref{eq:arc-surjections} are therefore isomorphisms, and hence $V$ is
classically free.
\end{proof}

\end{document}